\documentclass{article}

\usepackage[utf8]{inputenc}
\usepackage{amsmath,amsfonts,amssymb,amsthm,pb-diagram,picinpar,graphicx,color}
\usepackage{algorithm}
\usepackage{amscd}
\usepackage{algorithmic}
\usepackage{colortbl}
\usepackage{enumerate}
\usepackage{hyperref}
\usepackage{bm}
\usepackage{tikz}
\usepackage{pgf}
\usepackage[T1]{fontenc}
\usepackage{enumitem}
\usepackage{comment}
\usetikzlibrary{shapes,matrix,calc,arrows,decorations.markings,knots,patterns,intersections,cd}

\newcommand{\CC}{\mathbb{C}}
\newcommand{\ZZ}{\mathbb{Z}}
\newcommand{\PP}{\mathbb{P}}

\newcommand{\mcC}{\mathcal{C}}
\newcommand{\mcB}{\mathcal{B}}

\newcommand{\mcO}{\mathcal{O}}

\newcommand{\Supp}{\mathop{\mathrm{Supp}}\nolimits}

\newcommand{\Cmb}{\mathop{\mathrm{Cmb}}\nolimits}

\newtheorem{thm}{Theorem}[section]
\newtheorem{cor}[thm]{Corollary}
\newtheorem{prop}[thm]{Proposition}

\theoremstyle{definition}
\newtheorem{defin}[thm]{Definition}

\newtheorem{ex}[thm]{Example}

\theoremstyle{remark}
\newtheorem{rem}[thm]{Remark}

\renewcommand{\thesubparagraph}{\theparagraph.\@arabic\c@subparagraph}
\newcommand\blfootnote[1]{%
  \begingroup
  \renewcommand\thefootnote{}%
  \footnote{#1}%
  \endgroup
}

\begin{document}

\title{Examples of strong Ziegler pairs of conic-line arrangements of degree 7 and 8}

\author{Shinzo BANNAI\footnote{Partially supported by JSPS KAKENHI Grant Number JP23K03042} \, and Hiro-o TOKUNAGA\footnote{Partially supported by JSPS KAKENHI Grant Numbers JP23H00081 and JP24K06673}}

\date{\empty}
\maketitle

\begin{abstract}
A pair of plane curves with the same combinatorics is said to be
(a) a Zariski pair if the plane curves have different embedded topology, and
(b) a strong Ziegler pair if their Milnor algebra are not isomorphic.
We show that some examples of  Zariski pairs are also strong Ziegler pairs.
\end{abstract}

\section{Introduction}

Let $S = \CC[x, y, z]$.
Let $\mcB$ be a reduced plane curve in $\PP^2$ given by a homogeneous
polynomial $f_{\mcB}(x, y, z) \in S$. 
Let $\partial_x{f_{\mcB}}, \partial_yf_{\mcB}$, 
$\partial_zf_{\mcB}$ be partial derivatives of $f_{\mcB}$ by $x, y$ and $z$, respectively.
Let $J_{\mcB} := \langle \partial_xf_{\mcB}, \partial_yf_{\mcB}, \partial_zf_{\mcB}\rangle$ be the Jacobian ideal of
$f_{\mcB}$. Let  ${\mathrm {AR}}(\mcB):= \{(a, b, c) \in S^3 \mid 
a\partial_x{f_{\mcB}} + b\partial_yf_{\mcB} + c\partial_zf_{\mcB} =0\}$. Namely $\mathrm{AR}(\mcB)$ is the
graded $S$-module of Jacobian syzygies which has been widely and intensively studied for various curves, in particular, line,  conic-line and conic arrangements (e.g., \cite{bjp2025, cuntz_pokora, dimca_pokora22,  pokora23,  pokora-szemberg23} and references therein). In \cite{bjp2025}, the notion of a \emph{strong Ziegler pair} was defined as follows:
\begin{defin}[cf. \cite{adp2024}]
    Let, $\mcB_1, \mcB_2 \subset \PP^2$ be reduced plane curves. We say that $\mcB_1, \mcB_2$ form a \emph{strong Ziegler pair} if the combinatorics of the curves are equivalent, but the modules ${\mathrm {AR}}(\mcB_1)$ and ${\mathrm {AR}}(\mcB_2)$ are distinct.
\end{defin}
In \cite{bjp2025}, 
an explicit example
of a strong Ziegler pair was given. The example is a pair of conic-line arrangements of degree $8$.

 On the other hand, pairs of curves called \emph{Zariski pairs} were defined by E.~Artal in \cite{artal94} as follows:
\begin{defin}[cf. \cite{artal94}]
    Let, $\mcB_1, \mcB_2 \subset \PP^2$ be reduced plane curves. We say that $\mcB_1, \mcB_2$ form a \emph{Zariski pair} if the combinatorics of the curves are equivalent, but their  embedded topology is different, i.e., there exist no homeomorphisms $h:\PP^2\rightarrow\PP^2$ such that $h(\mcB_1)=\mcB_2$.
\end{defin}

The underlying themes in the study of strong Ziegler pairs and Zariski pairs are the same, and both aim to detect subtle differences in curves having fixed combinatorics. This is highlighted by the fact that the above mentioned example of a strong Ziegler pair given in \cite{bjp2025} was studied by the second author and colleagues in \cite{asstt21} in the context of Zariski pairs. This strongly suggests that Zariski pairs can be good candidates of strong Ziegler pairs. In fact the fist example of a Zariski pair consisting of sextic curves with six cusps turns out to be a strong Ziegler pair (see Section \ref{sec:zariski}). In view of the above facts, the authors think it is worthwhile to see how far the similarity goes by studying known Zariski pairs from the viewpoint of strong Ziegler pairs, especially in the case of conic-line arrangements.

In this note, we give new examples of strong 
Ziegler pairs of conic-line arrangement of degree $7$ and $8$; four examples for degree $7$ and an example for 
degree $8$. All of them have at most nodes, tacnodes and ordinary triple points as singularities.  We hope that these examples will be a small contribution to   Problem~4.12 raised in \cite{cuntz_pokora}.

Now we state our result.  The notation  $\Cmb_{ijk}$ for the  combinatorics of the curves is adopted from \cite{bty25} where the curves were originally studied in terms of Zariski pairs and will be described in Section~\ref{sec:degree7}.
\begin{thm}\label{thm:main}{
\begin{enumerate}
    \item[\rm(i)] There exist strong Ziegler pairs for conic-line arrangement of degree $7$ with combinatorcs $\Cmb_{123}, \Cmb_{124},
    \Cmb_{212}$ and $\Cmb_{224}$ 
    \item[\rm{(ii)}] There exists a strong Ziegler pair of conic-line arrangement of degree $8$ with
    the combinatoric described in \S~\ref{sec:degree8}. 
\end{enumerate}
}    
\end{thm}

Note that we adopt the definition of the \emph{combinatorics} of a curve from 
\cite{abst2023-1, survey, cogo-matei}, which is slightly different and more strict compared to  that of \cite{bjp2025}.  Nevertheless, for the  examples in this note, the plane
curves in each tuple have the same combinatorics in the sense of both \cite{abst2023-1, survey, cogo-matei} and \cite{bjp2025}. Note that all of them form examples of Zariski tuples, i.e., tuples of plane curves with the same combinatorics but different embedded topology.

\section*{Acknowldegements.} The authors thank Professor Takuro Abe and Professor Piotr Pokora for their comments on the first draft of
this article.

\section{Zariski's sextics}\label{sec:zariski}


Let us start with Zariski's example for a Zariski pair (\cite{zariski29, zariski31, zariski37}).

\begin{ex}\label{eg:zariski}{
Let $(\mcB_1, \mcB_2)$ be a pair of sextics in $\PP^2$ as follows:
\begin{enumerate}
    \item[(i)] $\mcB_i$ $(i = 1, 2)$ are irreducible and have $6$ cusps only 
    as their singularities.
    \item[(ii)] For $\mcB_1$, its six cusps  are on a conic, while there exists no such conic for $\mcB_2$.
\end{enumerate}
Then $(\mcB_1, \mcB_2)$ is a Zariski pair.  The differences in the embedded topology is detected through the fundamental gropus $\pi_1(\PP^2\setminus \mcB_i, \ast)$, where $\pi_1(\PP^2\setminus \mcB_1, \ast)\cong \ZZ/2\ZZ\ast \ZZ/3\ZZ$ and $\pi_1(\PP^2\setminus \mcB_2, \ast)\cong \ZZ/6\ZZ$ (see \cite{oka92}).}
\end{ex}

 As for explicit models, the first case is rather easy to find. In fact, we have such a sextic $\mcB_1$ given by
 \[
 (x^2 + y^2 + z^2)^3 + (x^3 + y^3 + z^3)^2 = 0
 \]
 for $\mcB_1$. On the other hand,  the second case is not so obvious and  such an example $\mcB_2$ can be found in \cite{oka92}, which is as follows:

\[
x^{6} - x^{4} y^{2} + \frac{1}{3} x^{2} y^{4} - \frac{1}{27} y^{6} + 2 x^{3} y^{2} z - 2 x^{4} z^{2} - \frac{5}{3} x^{2} y^{2} z^{2} - \frac{2}{9} y^{4} z^{2} + \frac{4}{3} x^{2} z^{4} + \frac{5}{9} y^{2} z^{4} - \frac{8}{27} z^{6} = 0
\]

We can check that $(\mcB_1, \mcB_2)$ satisfy the two conditions in Example~\ref{eg:zariski} easily (by computer system, e.g., Maple). The fundamental groups can also be calculated using SageMath and the package {\tt sirocco} \cite{sirocco} with the command {\tt fundamental\_groups} and we see that
$\pi_1(\PP^2\setminus \mcB_1, \ast) \not\cong \pi_1(\PP^2\setminus \mcB_2, \ast)$.  
Now we compute the minimal free resolutions of $M(\mcB_i)$ $(i = 1, 2)$ by using SageMath and the command {\tt graded\_free\_resolution} applied to
the Jacobian ideals $J_{\mcB_i}$ $(i = 1, 2)$. The results are as follows:

\begin{itemize}
    \item Resolution of $M(\mcB_1)$:
    \[
    0 \to S(-11)\oplus S(-12) \to S(-8)\oplus S(-10)^{\oplus 3} \to S(-5)^{\oplus 3} \to S(0).
    \]
   \item Resolution of $M(\mcB_2)$: 
   \[
   0 \to S(-11)^{\oplus 3} \to S(-9)^{\oplus 2}\oplus S(-10)^{\oplus 3} \to S(-5)^{\oplus 3} \to S(0).
   \]
\end{itemize}

 The above result shows 

 \begin{prop}\label{prop:zariski}{The Zariski pair in Example~\ref{eg:zariski} is a strong Ziegler pair.
 }   
 \end{prop}

 \begin{rem}
        In \cite{dimca_pokora23}, it was shown that any maximizing plane curve with ADE singularities are free. On the other hand,
    in \cite{artal-tokunaga00}, examples of Zariski pairs $(\mcB_1, \mcB_2)$ for maximizing sextics with singularities $A_{17} + 2A_1$ and $A_{11} + A_5 + 3A_1$ were given and the outputs by {\tt gareded\_free\_resolution} are all the same
    for curves given there, but we do not know if $\mathrm{AR}(\mcB_i)$ are different from each other.
    Note that
    there were miscalculation in \cite[Example 1.1, $B_2$ in Remark 2.1]{artal-tokunaga00}. Both of the cubics
    $C_1^{(2)}$ and $C_2^{(2)}$ are smooth and
    $C_1^{(2)} + C_2^{(2)}$ is not maximizing.
    In order to obtain the desired curve, we first
    consider the pencil of cubics $C_{\lambda} = C_1^{(2)} + \lambda C_2^{(2)}, \lambda \in \CC$ and choose
    $C_{\frac{97 + 63\sqrt{-3}}{146}}$ and
     $C_{\frac{35 + 45\sqrt{-3}}{73}}$. Then
     $C_{\frac{97 + 63\sqrt{-3}}{146}} +C_{\frac{35 + 45\sqrt{-3}}{73}}$ is a maximizing sextic with prescribed singularites, i.e., $A_{17} + 2A_1$.
     
 \end{rem}

\section{Proof of Theorem~\ref{thm:main} (i)}\label{sec:degree7}
%
%

In this section, we give examples of strong Ziegler pairs of conic-line arrangements of degree 7, based on the Zariski pairs that were studied in \cite{bty25}. As for the notation and terminologies
for our conic-line arrangements, we
use those given in \cite{bty25}.

In \cite{bty25} conic-line arrangements of
degree $7$ were studied. Combined with known results, the existence of Zariski pairs of conic-line arrangements with the following combinatorics was proved.

1. $\Cmb_{123}$. This consists of two conics $C$ and $D$ and three lines $L_1, L_2$ and $M$ as follows:
\begin{enumerate}
    \item[(i)] $C\pitchfork (L_1 + L_2)$, $L_1\cap L_2 \cap C = \emptyset$ and $D\pitchfork M$.
    \item[(ii)] $D$ is tangent to $C$ at one point and to $L_2$.
    \item[(iii)] $C\cap D\cap L_1$ consists of two points.
    \item[(iv)] $M$ passes through
    $L_1\cap L_2$ and tangent to $C$.
\end{enumerate}

2. $\Cmb_{124}$. This consists of two conics $C$ and $D$ and three lines $L_1, L_2$ and $M$ as follows:
\begin{enumerate}
    \item[(i)] $C\pitchfork (L_1 + L_2)$, $L_1\cap L_2 \cap C = \emptyset$ and $D\pitchfork M$.
    \item[(ii)] $D$ is tangent to $C$ at two points and to $L_1 + L_2$
    at two point in $(L_1 + L_2) \setminus C$.
    \item[(iii)] $M$ passes two points, one is in $L_1\cap C$ and the other in $L_2\cap C$.
\end{enumerate}

3. $\Cmb_{212}$. This consists of two conics $C_1$ and $C_2$ and three lines $M_1, M_2$ and $M_3$ as follows:
\begin{enumerate}
    \item[(i)] $C_1\pitchfork C_2$ and
    $M_1, M_2$ and $M_2$ are not concurrent.
    \item[(ii)] $M_1 \cap C_1 \cap C_2$
    consists of two points and $M_2$ and $M3$ are bitangent to $C_1 + C_2$.
\end{enumerate}

4. $\Cmb_{223}$. This consists of three conics $C_1, C_2$ and $D$ and a line $M$ as follows:
\begin{enumerate}
    \item[(i)] $C_1\pitchfork C_2$.
    \item[(ii)] $D$ passes through two points of $C_1\cap C_2$ and is tangent to both $C_1, C_2$.
    \item[(iii)] $M$
    is tangent to both $C_1, C_2$ and $M \pitchfork D$.
\end{enumerate}

5. $\Cmb_{224}$. This consists of three conics $C_1, C_2$ and $D$ and a line $M$ as follows:
\begin{enumerate}
    \item[(i)] $C_1\pitchfork C_2$ and
    $D$ inscribes $C_1 + C_2$ at four points.
    \item[(ii)] $M \cap C_1 \cap C_2$
    consists of two points and $M \pitchfork D$.
\end{enumerate}

\begin{figure}[H]
\centering
\begin{minipage}[b]{0.32\columnwidth}
    \centering
    \includegraphics[bb=00 00 700 700, width=4cm]{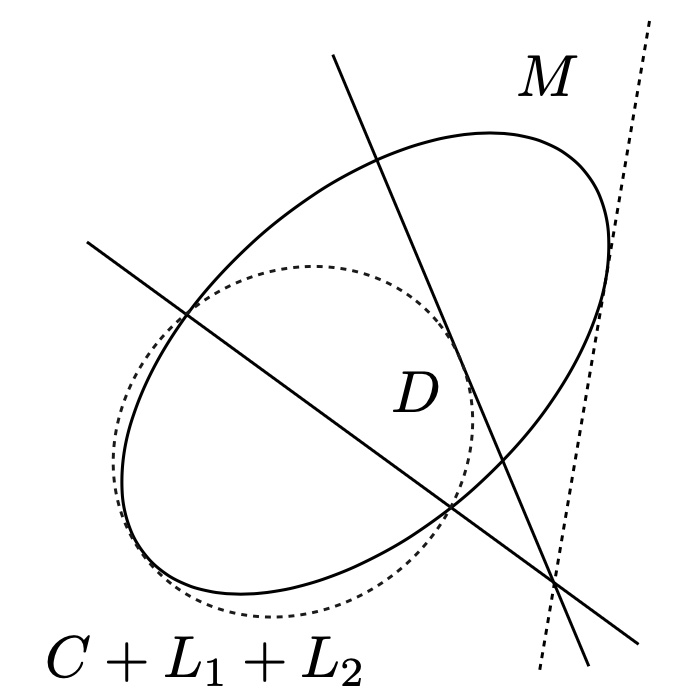}
    \caption{$\Cmb_{123}$}
\end{minipage}
\begin{minipage}[b]{0.32\columnwidth}
    \centering
    \includegraphics[bb=00 00 700 700, width=4cm]{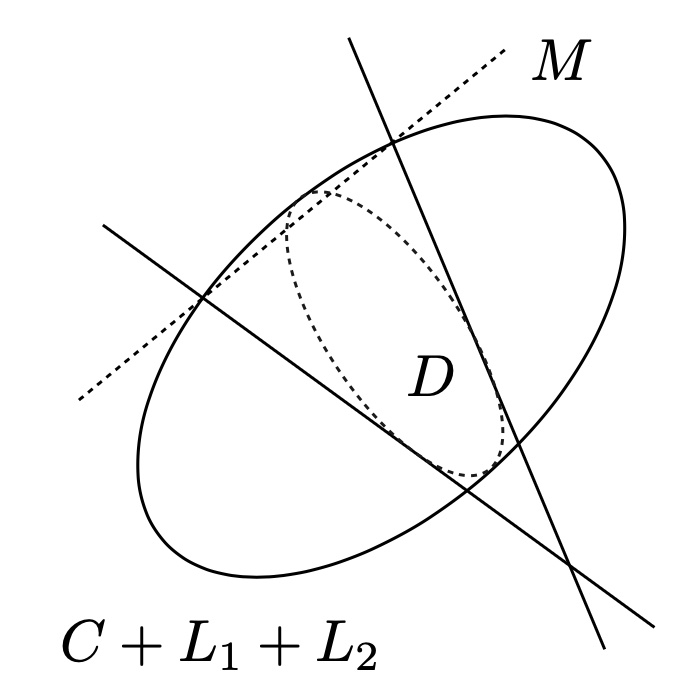}
    \caption{$\Cmb_{124}$}
\end{minipage}
\end{figure} 
\begin{figure}[H]
\centering
\begin{minipage}[b]{0.32\columnwidth}
    \centering
    \includegraphics[bb=00 00 700 700, width=4cm]{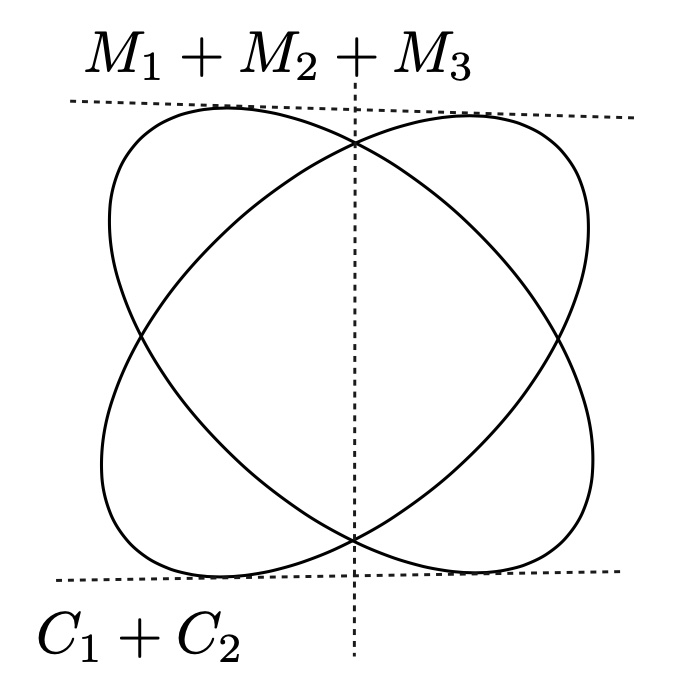}
    \caption{$\Cmb_{212}$}
\end{minipage}
\begin{minipage}[b]{0.32\columnwidth}
    \centering
    \includegraphics[bb=00 00 700 700, width=4cm]{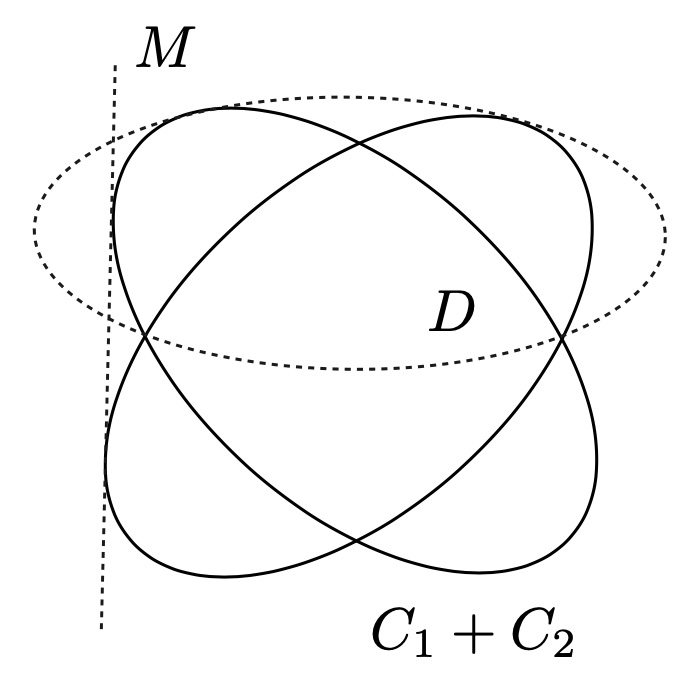}
    \caption{$\Cmb_{223}$}
\end{minipage}
\begin{minipage}[b]{0.32\columnwidth}
    \centering
    \includegraphics[bb=00 00 700 700, width=4cm]{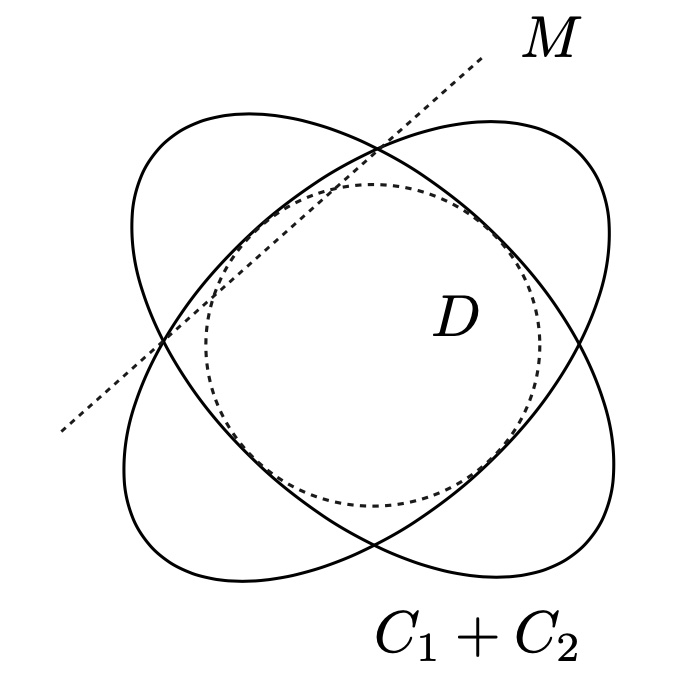}
    \caption{$\Cmb_{224}$}
\end{minipage}
\end{figure} 

As we have seen in \cite{tokunaga14, absst, bty25},  there exist Zariski pairs for
all of the five combinatorics. We will now give an explicit example of a
Zariski pair for each case, and check if they  form a strong  Ziegler pair.

\begin{ex}\label{ex:cmb123} $\Cmb_{123}$ (\cite[Example 4.7]{bty25}). 
Consider conics and lines as follows:
\[
C: -x^{2} + y z = 0, \quad L_1: 3x + y + 2z = 0, \quad L_2: -3x + y + 2z = 0,
\]
and
\[\begin{array}{c}
D: \left(12 b - 18\right) x^{2} + \left(-36 b + 51\right) x z + y z + \left(24 b - 34\right) z^{2} = 0, \\
M_{1}: 2bx + y + 2z = 0, \quad  M_{2}: -2bx + y + 2z = 0, \quad \text{where $\quad b = \sqrt{2}$}.
\end{array}
\]
Now put $\mcB_{1,1}= C + L_1 + L_2 + D + M_{1}$ and $\mcB_{1,2}= C + L_1 + L_2 + D + M_{2}$.
\end{ex}

\begin{ex}\label{ex:cmb124} $\Cmb_{124}$(\cite[Example 5.2]{tokunaga14}).  Let  $C, L_1$ and $L_2$ be those in Example~\ref{ex:cmb123}. Let 
\[
D: -\frac{9}{8} x^{2} + y z = 0, \quad M_{1}: -x + y - 2z = 0, \quad M_{2}: y - z = 0.
\]
Now put $\mcB_{2,1}= C + L_1 + L_2 + D + M_{1}$ and $\mcB_{2,2}= C + L_1 + L_2 + D + M_{2}$.    
\end{ex}

\begin{ex}\label{ex:cmb212} $\Cmb_{212}$(\cite[Example 4.16]{bty25}).  Let $C_1$ and $C_2$
be conics given by

\[
C_1:x^{2} + x y + y^{2} - \frac{27}{4} z^{2} = 0, \quad  C_2: 676 x^{2} + 764 x y + 676 y^{2} - 4563 z^{2} = 0.
\]
Let $M_0, M_1, M_2$ and $M_3$ be four lines given by
\[
M_0: y = 0, \,\, M_1: 15x + 8y - 39z = 0, \,\, M_2: 15x + 8y + 39z = 0, \,\, M_3: 8x + 15y - 39z = 0.
\]
Now put $\mcB_{3,1}= C_1 + C_2 + M_0 + M_1 + M_{2}$ and $\mcB_{3,2}= C_1 + 
C_2 + M_0 + M_1 + M_3$.   
    
\end{ex}

\begin{ex}\label{ex:cmb223}(\cite[Introduction]{absst}) Let $C_1$, $C_2$ and $D$
be conics given by
\[
C_1: -x^{2} + y z, \quad C_2: -10 x y + y^{2} + 25 y z - 36 z^{2}, \quad D: -\frac{5}{4} x^{2} + 2 x z + y z - 3 z^{2},
\]

Let $M_1, M_2, M_3$ and $M_4$ be four lines given by
\[
M_1: -\frac{32}{5} x + y + \frac{256}{25} z = 0,\, 
M_2: y = 0,
\]
\[M_3: -10 x + y + 25 z = 0, \, 
M_4: -\frac{18}{5} x + y + \frac{81}{25} z = 0.
\]
Now put $\mcB_{4,i}= C_1 + C_2 + D + M_i$ 
($i = 1, 2, 3, 4$).     
\end{ex}

\begin{ex}\label{ex:cmb224}(\cite[Example 5.2]{tokunaga14}) 
Let $C_1, C_2$ and $D$ 
be conics given by
\[
C_1: -x^{2} + y z + 2 z^{2} = 0, \quad  C_2: x^{2} + y^{2} - 2 y z - 4 z^{2} = 0, \quad D: -\frac{1}{2} x^{2} + y z + 2 z^{2} = 0.
\]
Let $M_1$ and $M_2$ be lines given by
\[
M_1: -x + y = 0, \quad M_2: -3x + y + 4z = 0.
\]
 Now put $\mcB_{5,1}= C_1 + C_2 + D + M_1$ and $\mcB_{5,2}= C_1 + 
C_2 + D + M_2$.     
\end{ex}

\begin{prop}\label{prop:ziegler}{Let $(\mcB_{i,1}, \mcB_{i,2})$ $(i = 1, 2, 3, 5)$ be pairs of conic-line arrangements as above. Then 
$(\mcB_{i, 1}, \mcB_{i,2})$ $(i = 1, 2, 3, 5)$ form strong Ziegler pairs.
}
\end{prop}

\proof For each pair $(\mcB_{i, 1}, \mcB_{i, 2})$, $\mcB_{i,1}$ and 
$\mcB_{i, 2}$ have the same combinatorics. Now we compute the minimal
resolution of the associated Milnor algebras for each case by using
the SageMath command {\tt graded\_free\_resolution}. Then our statement
follows from the following:
\begin{itemize}
\item Resolution for $M(\mcB_{i,1})$ ($i = 1, 2,3$) :
\[
0 \to S(-12) \to S(-9)\oplus S(-10) \oplus S(-11) \to S(-6)^{\oplus 3} \to S(0).
\]
\item Resolution for $M(\mcB_{i,2})$ ($i = 1, 2, 3$) :
\[
0 \to S(-11)^{\oplus 2} \to  S(-10)^{\oplus 4} \to S(-6)^{\oplus 3} \to S(0).
\]
\item Resolution for $M(\mcB_{5,1})$ :
\[
0 \to S(-13) \to S(-9)\oplus S(-10) \oplus S(-12) \to S(-6)^{\oplus 3} \to S(0).
\]
\item Resolution for $M(\mcB_{5,2})$ :
\[
0 \to S(-11)\oplus S(-12) \to  S(-10)^{\oplus 3}\oplus S(-11) \to S(-6)^{\oplus 3} \to S(0).
\]
\end{itemize}
\endproof

\begin{rem}\begin{enumerate}
    \item[(i)] The example of a strong Ziegler pair given in \cite{bjp2025} is studied in \cite{asstt21} from
    the view point of Zariski pairs. In fact, it gives a Zariski pair, which
    can be regarded as a degeneration of Zariski pairs of conic arrangements
    studied by Namba and Tsuchihashi in \cite{Namba-Tsuchihashi}. With
    {\tt graded\_free\_resolution}, we can also check that Namba-Tsuchihashi's
    Zariski pair for conic arrangements gives a strong Ziegler pair.

    \item[(ii)] For $M(\mcB_{4,i})$ ($i = 1, 2, 3, 4$), we have resolutions:
\[
0 \to S(-12) \to S(-10)^{\oplus 3} \to S(-6)^{\oplus 3} \to S(0)
\]
for all $i = 1, 2, 3, 4$. The pair $(\mcB_{4,i}, \mcB_{4,j})$
$\{i,j\} =\{1, 3\}, \{1, 4\}, \{2, 3\}, \{2, 4\}$
are known to be  Zariski pairs by \cite[Section4.2]{absst} but {\tt graded\_free\_resolution} outputs the same minimal resolutions for plane curves in the example.  We have not checked if ${\mathrm{AR}}(\mcB_i)$ $(i=1, 2)$ are distinct as modules. Also, this pair is exceptional among the five cases of degree 7 that we have presented in the sense that this pair is the only one whose fundamental groups $\pi_1(\PP^2\setminus \mcB_i, \ast)$ $(i=1, 2)$ are abelian and isomorphic. The other four cases have non-abelian and non-isomorphic fundamental groups (see \cite{bty25}). We will revisit this pair in the case of degree 8.

\item[(iii)]
In \cite{adp2024}, a new hierarchy for projective plane curves \lq{\it type $t(\mcB)$}' is introduced (\cite[Definition 1.2]{adp2024}). For our examples, 
$t(\mcB_{i, 1}) = 1, t(\mcB_{i,2}) = 2$ ($i = 1, 2, 3, 5$).

\end{enumerate}
    
\end{rem}

\section{Proof of Theorem~\ref{thm:main} (ii)}\label{sec:degree8}
In this section, we will give an example of a Zariski and strong Ziegler pairs for conic-line arrangements of degree $8$.
Before we go on to our example, let us explain \lq Splitting type\rq\, for
a reduced plane curve of the form $\mcB + \mcC$, where (i) both $\mcB$ and $\mcC$ are reduced, (ii) $\mcB$ and $\mcC$ have no common component and
(iii) $\deg \mcB$ is even.

\subsection{Splitting type}

 The notion of splitting type arose from the notion of \emph{splitting curves} described below, which can be considered as the simplest example of a splitting invariant. 
Let $\mcB$ be a plane curve of even degree and let 
$f'_{\mcB} : S'_{\mcB} \to \PP^2$ be the double cover of $\PP^2$ with branch locus
$\mcB$. Let $\mu_{\mcB}: S_{\mcB} \to S'_{\mcB}$ be the canonical resolution fitting into 
the following commutative diagram:
\[
\begin{CD}
S'_{\mcB} @<{\mu_{\mcB}}<< S_{\mcB} \\ 
@V{f'_{\mcB}}VV                 @VV{f_{\mcB}}V  \\ 
\PP^2@<<{q}< \widehat{\PP^2} \\ 
\end{CD}
\]
where  $q$ is a composition of a finite 
 number of blowing-ups so that the branch locus becomes smooth (See \cite{horikawa}
 for the canonical resolution) and $f_{\mcB}: S_{\mcB} \to \widehat{\PP^2}$ is the induced
 double cover.

Put $\tilde{f}_{\mcB} = f'_{\mcB}\circ \mu_{\mcB} = q\circ f_{\mcB}$.
Let $C$ be an irreducible plane curve not contained in $\mcB$. The pull-back $\tilde{f}_{\mcB}^*C$ is of the form either
\begin{center}
  (a) $C^+ + C^- + E$ \quad or \quad (b) $\tilde C + E$
\end{center}
where $C^\pm$ and $\tilde C$ are irreducible with $\tilde{f}_{\mcB}(C^{\pm}) = 
\tilde{f}_{\mcB}(\tilde C)= C$ 
and $\Supp(E)$ is contained in the exceptional
set of $\mu_{\mcB}$.

\begin{defin}\label{def:splitting}{\rm
We say that an irreducible plane curve $C$ is a \textit{splitting curve} with respect to $B$
if the case (a) above holds for~$C$. 
}
\end{defin}

We now give the definition of splitting types.

\begin{defin}[cf. {\cite{bannai16, bst24}}]\label{def:splitting_type}
Let $f'_{B}:S'_B \to\PP^2$ be a double cover branched along a reduced 
plane curve $B$, and
let $D_1, D_2\subset\PP^2$ be two irreducible curves such that ${f'}_{B}^\ast D_i$ are reducible  with irreducible decomposition ${f'}_{B}^\ast D_i=D_i^++D_i^-$.
For integers $m_1\leq m_2$, we say that the triple $(D_1, D_2;B)$ has the {\it splitting type} $(m_1, m_2)$ if for a suitable choice of labels $D_1^+\bar{\cdot} D_2^+=m_1$ and $D_1^+\overline{\cdot} D_2^-=m_2$. Here $\overline{\cdot}$ denotes
the sum of local intersection multiplicities at points over $\PP^2\setminus B$. We abuse notation and use $(D_1, D_2;B)$ to denote both the triple and its splitting type as follows:
\[
(D_1, D_2;B) = (m_1, m_2).
\]

\end{defin}

\begin{rem}{\rm In the study of vector bundles $E$ of rank $r$ over $\PP^2$, the terminology \rq splitting type\rq\, is also used. When we restrict
$E$ to a line, $E$ splits into a direct sum of line bundles
$E \cong \oplus_{i= 1}^r \mcO(a_i)$. In this context \lq\lq the splitting type of $E$'' refers to the sequence of integers 
$(a_1, \ldots, a_r)$, which is different from the one in Definition~\ref{def:splitting_type}.
}
\end{rem}
The following proposition enables us to distinguish the embedded topology of plane curves by the splitting type, which also holds under the slightly modified setting as above.

\begin{prop}[cf. {\cite[Proposition~2.5]{bst24}}]\label{prop:splitting_type}
Let $f'_{B_i} :S'_{B_i}\to\PP^2$ $(i=1,2)$ be two double covers branched along plane curves $B_i$, respectively.
For each $i=1,2$, let $D_{i1}$ and $D_{i2}$ be two irreducible plane curves such that ${f'}_{B_i}^\ast D_{ij}$ are reducible with irreducible decomposition
${f'}_{B_i}^\ast D_{ij}=D_{ij}^++D_{ij}^-$.
Suppose that  $D_{i1}$ and $D_{i2}$ intersect transversely over $\PP^2\setminus B_i$, and that $(D_{11},D_{12};B_1)$ and $(D_{21}, D_{22};B_2)$ have distinct splitting types.
Then there is no homeomorphism $h:\PP^2\to\PP^2$ such that $h(B_1)=B_2$ and $\{h(D_{11}), h(D_{12})\}=\{D_{21}, D_{22}\}$.
\end{prop}

For a proof see \cite{bst24}. 

\begin{rem}
    The splitting type has been used in \cite{absst, bannai16, BKMT22, bmsty-poncelets, bannai-tokunaga17} to distinguish the embedded topology of the curves considered there. 
\end{rem}

Proposition \ref{prop:splitting_type} can be generalized to cases involving a larger number of splitting curves. For our purpose, the following form for three splitting curves will be used later.

\begin{cor}\label{cor:splitting_type} Let $\mcB_i$ $(i =1, 2)$ as in Proposition~\ref{prop:splitting_type}. 
For each $i=1,2$, let $D_{i1}$,  $D_{i2}$ and $D_{i3}$ be three irreducible plane curves such that ${f'}_{B_i}^\ast D_{ij}$ are reducible with irreducible decomposition
${f'}_{B_i}^\ast D_{ij}=D_{ij}^++D_{ij}^-$.
Suppose that  $D_{ij}$ and $D_{ik}$ $(1 \le j < k \le 3)$ intersect transversely over $\PP^2\setminus B_i$, and that the sets of splitting types 
\begin{eqnarray*}
 && \{(D_{11},D_{12};B_1),(D_{11},D_{13};B_1), (D_{12},D_{12};B_1)\}, \\
 &&
 \{(D_{21},D_{22};B_2),(D_{21},D_{23};B_2), (D_{22},D_{23};B_1)\}.
\end{eqnarray*}
are distinct, i.e., they are distinct as multi-sets.
Then there is no homeomorphism $h:\PP^2\to\PP^2$ such that $h(B_1)=B_2$ and $\{h(D_{11}), h(D_{12}), h(D_{13})\}=\{D_{21}, D_{22}, D_{23}\}$.
\end{cor}

\proof With Proposition~\ref{prop:splitting_type} we apply \cite[Proposition 1.2]{bgst17} to our case as
$A=$ the set of possible splitting types and our statement follows. \qed

\subsection{A Zariski triple and  strong Ziegler 
pairs for conic-line arrangements of degree \texorpdfstring{$8$}{}}

In this section, we give an example of a strong Ziegler pare based on a Zariski triple of conic-line arrangemets of degree 8. The combinatorics of these arrangements is obtained by adding an additional bitangent line to $\Cmb_{223}$ described in the previous section. More precisely, the conic-line arrangements that we consider consist of the following curves
\begin{itemize}
    \item $C_1, C_2$ : smooth conics intersecting transversely at four points $p_1, \ldots, p_4$.
    \item $M_1, \ldots, M_4$ : the four bitangent lines of $C_1+C_2$.
    \item $D$ : a weak contact conic to $C_1+C_2$, i.e. a conic passing through two of $p_1, \ldots, p_4$ that is tangent to both $C_1$ and $C_2$.
\end{itemize}
and are of the form
\[
\mcB = C_1+C_2+D+M_i+M_j
\]
with the following assumptions on the combinatorics :
\begin{itemize}
    \item[(i)] $C_i \cap M_j \cap D =\emptyset$ for $i=1, 2, j = 1, \ldots, 4$. Namely, the tangent point of $C_i, M_j$ and $C_i, D$ are distinct. (This condition is necessary satisfied if $D$ is irreducible.)
    \item[(ii)] $D\cap M_i \cap M_j=\emptyset$. Namely, $D, M_i, M_j$ do not intersect in one point.
\end{itemize}
\begin{figure}[H]
\centering
\begin{minipage}[b]{0.32\columnwidth}
    \centering
    \includegraphics[bb=00 00 700 700, width=4cm]{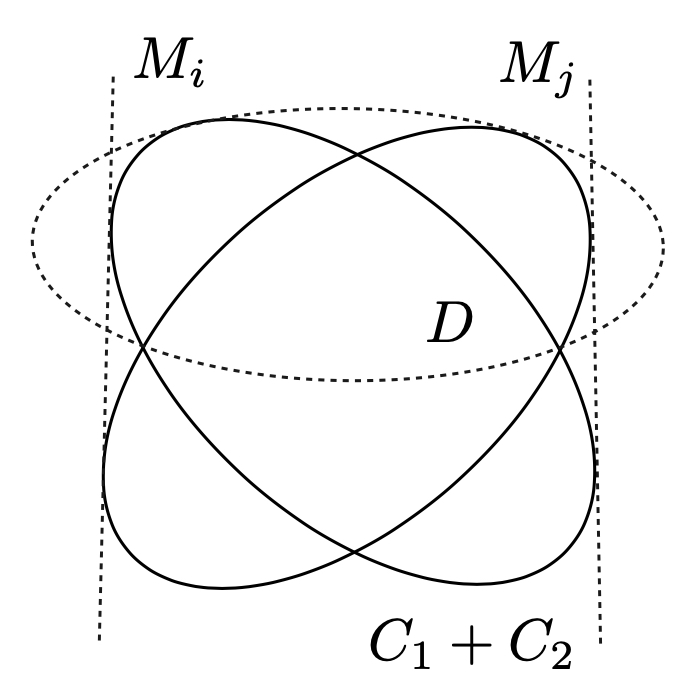}
    \caption{$\Cmb_{223}+$another bitangent}
    \label{fig:cmb_223} 
\end{minipage}
\end{figure}

A conic-line arrangement with the above combinatorics has $7$ nodes, $6$ tacnodes and $2$ ordinary triple points as singularities, as depicted in Figure \ref{fig:cmb_223}. Note that the node arising from the intersection point of $M_i$ and $M_j$ is at infinity and is not depicted.

\begin{prop} Let $C_1, C_2, D, M_1,\ldots, M_4$ be as in Example~\ref{ex:cmb223}. Put
\begin{align*}
        \mcB_1 = C_1+C_2+D+M_1+M_2\\
        \mcB_2 = C_1+C_2+D+M_1+M_3\\
        \mcB_3 = C_1+C_2+D+M_3+M_4.
    \end{align*}
Then
\begin{enumerate}
    \item[\rm{(i)}] $(\mcB_1, \mcB_2, \mcB_3)$ is a
    Zariski triple
    \item[\rm{(ii)}] The  minimal free resolutions 
    for $M(\mcB_i)$ ($i = 1, 2, 3$) are
    \begin{itemize}
    \item Resolution of $M(\mcB_1)$ : 
    \[
    0 \to S(-14) \to  S(-12)^{\oplus 2}\oplus S(-11) \to S(-7)^{\oplus 3} \to S(0).
    \]
 \item Resolutions of $M(\mcB_2)$ and $M(\mcB_3)$ :  
    \[
    0 \to S(-13)^{\oplus 3} \to  S(-12)^{\oplus 5} \to S(-7)^{\oplus 3} \to S(0).
    \]
    In particular, $(\mcB_1, \mcB_2)$ and 
    $(\mcB_1, \mcB_3)$ are strong Ziegler pairs.
\end{itemize}
\end{enumerate}
\end{prop}
\begin{proof}
(i)
 By \cite[Section 4.2]{absst}, we have
\[
(D, M_i; C_1+C_2)=\begin{cases} (0, 2) & i=1, 2, \\
(1, 1) & i=3, 4.
\end{cases}
\]
Hence $(\mcB_1, \mcB_2, \mcB_3)$ is a Zarsiki
triple by Corollary~\ref{cor:splitting_type}.

(ii) Our statement is immediate by using
 {\tt graded\_free\_resolution}.
\end{proof}

\begin{rem} For the above example, $t(\mcB_1) = 2$ and $t(\mcB_2) = t(\mcB_3) = 3$.   
\end{rem}
\begin{rem} 
The existence of "special curves" such as the conic passing through certain singular points of the arrangement as described above seems to affect the graded free resolution. 
\end{rem}

\begin{rem}
    It is interesting to see that the Zariski pair with combinatorics $\Cmb_{223}$ did not yield a strong Ziegler pair decisively, but it leads to a strong Ziegler pair. Adding an additional bitangent line seems to have exposed the underlying differences of the pair in the form of differences in the resolution.  It may be interesting to investigate other Zariski pairs and see if something similar occurs.  
\end{rem}

\bibliographystyle{spmpsci}
\bibliography{biblio.bib}

\noindent Shinzo BANNAI\\
Department of Applied Mathematics, \\
Faculty of Science, \\
Okayama University of Science, \\
1-1 Ridai-cho, Kita-ku, Okayama 700-0005 JAPAN
\\
{\tt bannai@ous.ac.jp}

\noindent Hiro-o TOKUNAGA\\
Department of Mathematical  Sciences, Graduate School of Science, \\
Tokyo Metropolitan University, 1-1 Minami-Ohsawa, Hachiohji 192-0397 JAPAN \\
{\tt tokunaga@tmu.ac.jp}
\end{document}